\documentclass[12pt]{amsart}
\usepackage[latin1]{inputenc}
\usepackage{color}
\usepackage{enumerate}
\usepackage{amssymb}
\usepackage[all]{xy}
\usepackage{hyperref}

\usepackage{enumerate}
\usepackage{tikz}
\usepackage{subcaption}

\newtheorem{thm}{Theorem}[section]
\newtheorem{cor}[thm]{Corollary}
\newtheorem{lem}[thm]{Lemma}

\newtheorem{prop}[thm]{Proposition}
\theoremstyle{definition}
\newtheorem{defn}[thm]{Definition}

\theoremstyle{remark}
\newtheorem{rem}[thm]{Remark}
\newtheorem{prob}{Problem}
\newtheorem{example}[thm]{Example}
\numberwithin{equation}{section}
\newcommand{\norm}[1]{\left\Vert#1\right\Vert}
\newcommand{\abs}[1]{\left\vert#1\right\vert}
\newcommand{\set}[1]{\left\{#1\right\}}

\newcommand{\R}{\mathbb{R}}

\DeclareMathOperator{\lat}{lat}

\usepackage{color}
\usepackage{enumerate}
\begin{document}
\setcounter{tocdepth}{1}


\title{Lattice embeddings in free Banach lattices over lattices}
\author[Avil\'es]{Antonio Avil\'es}
\address[Avil\'es]{Universidad de Murcia, Departamento de Matem\'{a}ticas, Campus de Espinardo 30100 Murcia, Spain
	\newline
	\href{https://orcid.org/0000-0003-0291-3113}{ORCID: \texttt{0000-0003-0291-3113} } }
\email{\texttt{avileslo@um.es}}

\author[Mart\'inez-Cervantes]{Gonzalo Mart\'inez-Cervantes}
\address[Mart\'inez-Cervantes]{Universidad de Murcia, Departamento de Matem\'{a}ticas, Campus de Espinardo 30100 Murcia, Spain
	\newline
	\href{http://orcid.org/0000-0002-5927-5215}{ORCID: \texttt{0000-0002-5927-5215} } }	

\email{gonzalo.martinez2@um.es}

\author[Rodr\'iguez Abell\'an]{Jos\'e David Rodr\'iguez Abell\'an}
\address[Rodr\'iguez Abell\'an]{Universidad de Murcia, Departamento de Matem\'{a}ticas, Campus de Espinardo 30100 Murcia, Spain 	\newline
	\href{https://orcid.org/0000-0002-2764-0070}{ORCID: \texttt{0000-0002-2764-0070} }}

\email{josedavid.rodriguez@um.es}

\author[Rueda Zoca]{Abraham Rueda Zoca}
\address[Rueda Zoca]{Universidad de Murcia, Departamento de Matem\'{a}ticas, Campus de Espinardo 30100 Murcia, Spain
	\newline
	\href{https://orcid.org/0000-0003-0718-1353}{ORCID: \texttt{0000-0003-0718-1353} }}
\email{\texttt{abraham.rueda@um.es}}
\urladdr{\url{https://arzenglish.wordpress.com}}



\begin{abstract} In this article we deal with the free Banach lattice generated by a lattice and its behavior with respect to subspaces. In general, any lattice embedding $i\colon \mathbb{L} \longrightarrow \mathbb{M}$ between two lattices $\mathbb{L} \subseteq \mathbb{M}$ induces a Banach lattice homomorphism $\hat \imath\colon FBL \langle \mathbb{L} \rangle \longrightarrow FBL \langle \mathbb{M}\rangle$ between the corresponding free Banach lattices.
We show that this mapping $\hat \imath$ might not be an isometric embedding neither an isomorphic embedding. In order to provide sufficient conditions for $\hat \imath$ to be an isometric embedding we define the notion of locally complemented lattices and prove that, if $\mathbb L$ is locally complemented in $\mathbb M$, then $\hat \imath$ yields an isometric lattice embedding from $FBL\langle\mathbb L\rangle$ into $FBL\langle\mathbb M\rangle$. We provide a wide number of examples of locally complemented sublattices and, as an application, we obtain that every free Banach lattice generated by a lattice is lattice isomorphic to an AM-space or, equivalently, to a sublattice of a $C(K)$-space.

Furthermore, we prove that $\hat \imath$ is an isomorphic embedding if and only if it is injective, which in turn is equivalent to the fact that every lattice homomorphism $x^*\colon \mathbb{L} \longrightarrow [-1,1]$ extends to a lattice homomorphism $\hat x^*\colon \mathbb{M} \longrightarrow [-1,1]$. Using this characterization we provide an example of lattices $\mathbb{L} \subseteq \mathbb{M}$ for which $\hat \imath$ is an isomorphic not isometric embedding.
\end{abstract}

\maketitle

\section{Introduction}

In the last years, several free structures concerning Banach lattices have been considered and have attracted the attention of many researchers. On the one hand, B.~de Pagter and A.~W.~Wickstead considered in \cite{dPW15} the free Banach lattice generated by a set. This concept was later extended by A.~Avil\'es, J.~Rodr\'iguez and P.~Tradacete in \cite{ART18}, who introduced the free Banach lattice generated by a Banach space (the free Banach lattice generated by a set $A$ is lattice isometric to the free Banach lattice generated by the Banach space $\ell_1(A)$ \cite[Corollary 2.9]{ART18}). Finally, a different approach was considered by A. Avil\'es and J.~D.~Rodr\'iguez Abell\'an in \cite{ARA18} by considering the free Banach lattice generated by a lattice. 

Apart from being natural definitions in the framework of category theory, the concepts of free Banach lattices have shown to be useful in order to provide relevant examples and counterexamples in the theory of Banach lattices. For instance, in \cite{ART18} a free Banach lattice over a Banach space is exhibited as an example of a Banach lattice which is weakly compactly generated as a Banach lattice but not as a Banach space, answering an open question posed by J.~Diestel. Moreover, in \cite{dmrr2}  free Banach lattices were used in order to exhibit examples of lattice homomorphisms which do not attain their norm. See \cite{amr20_2,apr18,ara19,dmrr,jjttt} for background on free Banach lattices.

It is natural to ask whether the inclusion of two objects induces an inclusion of the free Banach lattices generated by them. In \cite{dPW15} it is shown that if $B$ is a non-empty subset of $A$, then the free Banach lattice generated by $B$, $FBL(B)$, is isometric to a sublattice of $FBL(A)$. 
This is no longer true for the free Banach lattice generated by a Banach space. Namely, the problem of determining when does an isomorphic embedding of Banach spaces yield a lattice isomorphic embedding between the corresponding free Banach lattices was studied by T. Oikhberg, M. Taylor, P. Tradacete and V. Troitsky, providing a complete answer in terms of operators taking values on $\ell_1^n$ (see Theorem \ref{theo:peter}).
In general,  if $F$ is a closed subspace of the Banach space $E$ and we denote by $FBL[F]$ and $FBL[E]$ the free Banach lattices generated by $F$ and $E$, respectively, then $FBL[F]$ is isometric to a sublattice of $FBL[E]$ whenever $F$ is a $1$-complemented subspace of $E$ \cite[Corollary 2.8]{ART18} or whenever $E=F^{**}$ \cite[Lemma 2.4]{dmrr}. Notice that in the latest case $F$ is locally complemented in $E$; in general, this condition implies that $FBL[F]$ is isometric to a locally complemented sublattice of $FBL[E]$ (see \cite[Corollary 4.2]{amr2021} for a generalized version of this result).

In this paper we focus on the analogous problem for the free Banach lattice generated by a lattice $\mathbb{L}$, denoted by $FBL \langle \mathbb{L} \rangle$. This problem has been previously considered in \cite[Lemma 5.3]{ARA18}, where it was proved that if $\mathbb{L} \subseteq \mathbb{M}$ are linearly ordered sets, then $FBL \langle \mathbb{L} \rangle$ is isomorphic to a sublattice of $FBL \langle \mathbb{M} \rangle$ (for an isometric version see \cite[Lemma 3.6]{jdthesis}). Furthermore, \cite[Proposition 4.2]{amrr2020} asserts that $FBL\langle\mathbb L\rangle$ is a $1$-complemented Banach sublattice of $FBL\langle\mathbb M\rangle$ whenever $\mathbb L$ is complemented in $\mathbb M$. We wonder how can this condition be relaxed in order to obtain, not 1-complementation, but an isometric inclusion.

\bigskip 

Let us now describe the content of the paper. In Section \ref{section:notation} we introduce necessary notation and preliminary results among which Theorem \ref{theo:condipedro} provides a sufficient condition for a pair of lattices $\mathbb L\subseteq\mathbb M$ to satisfy that $FBL\langle\mathbb L\rangle$ is an isometric sublattice of $FBL\langle\mathbb M\rangle$. In Section \ref{section:main} we introduce, motivated by the definition of locally complemented Banach spaces \cite{kalton84}, the notion of \textit{locally complemented sublattices} in Definition \ref{defi:lattlocomp}. Next, we prove in Theorem \ref{TheoremLocComExtDual} a result of extension of lattice homomorphisms which allows us to obtain the desired result, namely, that if $\mathbb L$ is a locally complemented sublattice of $\mathbb M$, then $FBL\langle \mathbb L\rangle$ is  isometric to a sublattice of $FBL\langle\mathbb M\rangle$ (see Corollary \ref{cor:fbllatisubs}). We conclude Section \ref{section:main} with several distinguished examples of pairs of lattices $\mathbb L\subseteq \mathbb M$ for which $\mathbb L$ is locally complemented in $\mathbb M$ such as $\mathbb M$ being linearly ordered regardless $\mathbb L$, or $\mathbb L$ being an ideal in $\mathbb M$ (see Proposition \ref{PropositionExamplesLocComp}). As a consequence of one of this examples we obtain one of the main results of the article, which asserts that $FBL\langle \mathbb{L} \rangle$ is an AM-space for every lattice $\mathbb{L}$ (Theorem \ref{TheoAMspace}).

In Section \ref{SectionIsomorphic} we focus on the ``isomorphic problem'', i.e.~we study when does the inclusion of $\mathbb L$ into $\mathbb M$ induce a lattice isomorphic embedding of  $FBL\langle\mathbb L\rangle$ into $FBL\langle\mathbb M\rangle$. We prove that the map $\hat{i}\colon FBL\langle\mathbb L\rangle\longrightarrow FBL\langle\mathbb M\rangle$ induced by the canonical inclusion $i\colon \mathbb L \longrightarrow \mathbb M$ is an into isomorphism if, and only if, it is injective, which in turn is equivalent to the fact that every lattice homomorphism $x^* \colon \mathbb L\longrightarrow [-1,1]$ admits an extension $\hat x^* \colon \mathbb M\longrightarrow [-1,1]$. Two main ingredients are needed here: on the one hand, a Banach lattice identification of $FBL\langle\mathbb L\rangle$ with a certain $C(K)$ space when $\mathbb L$ has a maximum and a minimum element \cite[Theorem 2.7]{amrr2020}; on the other hand, the fact that $\mathbb L$ is locally complemented in a lattice with maximum and minimum (see Proposition \ref{PropositionExamplesLocComp}).

Finally, using the results obtained in Sections \ref{section:main} and \ref{SectionIsomorphic}, we provide in Section \ref{section:isomonotisome} an example of (finite) lattices $\mathbb{L} \subseteq \mathbb{M}$ for which the map  $\hat{\imath}\colon FBL\langle\mathbb L\rangle\longrightarrow FBL\langle\mathbb M\rangle$ is an isomorphic not isometric embedding.

\section{Background and preliminary results}\label{section:notation}

Given a Banach space $E$, we denote by $FBL[E]$ the free Banach lattice generated by $E$. This Banach lattice, introduced in \cite{ART18}, contains an isometric copy of $E$ with the property that every bounded operator from $E$ to any Banach lattice $X$ can be extended to a lattice homomorphism from $FBL[E]$ to $X$ preserving the norm. 
In this article we focus on the notion of the free Banach lattice generated by a lattice $\mathbb{L}$ introduced in \cite{ARA18}.

\begin{defn}\label{FBLGBL}
Given a lattice $\mathbb{L}$, a \textit{free Banach lattice over} or \textit{generated by $\mathbb{L}$} is a Banach lattice $F$ together with a norm-bounded lattice homomorphism $\delta_{\mathbb{L}}\colon \mathbb{L} \longrightarrow F$ with the property that for every Banach lattice $X$ and every norm-bounded lattice homomorphism $T \colon \mathbb{L} \longrightarrow X$ there is a unique Banach lattice homomorphism $\hat{T}\colon F \longrightarrow X$ such that $T = \hat{T} \circ \delta_{\mathbb{L}}$ and $\| \hat{T} \| = \| T \|$.
$$\xymatrix{\mathbb{L}\ar_{\delta_{\mathbb{L}}}[d]\ar[rr]^T&&X\\
F\ar_{\hat{T}}[urr]&& }$$
\end{defn}

Here, the norm of $T$ is $\norm{T} := \sup \set{\norm{T(x)} : x \in \mathbb{L}}$, while the norm of $\hat{T}$ is the usual one for Banach spaces.

This definition determines a Banach lattice that we denote by $FBL\langle \mathbb{L}\rangle$ in an essentially unique way. When $\mathbb{L}$ is a distributive lattice (which is a natural assumption in this context, see \cite[Section~3]{ARA18}) the function $\delta_{\mathbb{L}}$ is injective and, loosely speaking, we can view $FBL\langle \mathbb{L} \rangle$ as a Banach lattice which contains a subset lattice-isomorphic to $\mathbb{L}$ in a way that its elements work as free generators modulo the lattice relations on $\mathbb{L}$.
Since every free Banach lattice generated by a lattice can be seen as a free Banach lattice generated by a distributive lattice 
\cite[Proposition 3.2]{ARA18}, for our purpose there is no loss of generality in considering only distributive lattices. For that reason, all lattices considered in this text are assumed to be  distributive.

In order to give an explicit description of it as a space of functions, define $$\mathbb{L}^{\ast} = \set {x^{\ast}\colon \mathbb{L} \longrightarrow [-1,1] : x^{\ast} \text{ is a lattice homomorphism}}.$$ For every $x \in \mathbb{L}$ consider the evaluation function $\delta_x \colon \mathbb{L}^{\ast} \longrightarrow \mathbb{R}$ given by $\delta_x(x^{\ast}) = x^{\ast}(x)$, and for $f \in \mathbb{R}^{\mathbb{L}^{\ast}}$, define $$ \norm{f} = \sup \set{\sum_{i = 1}^n \abs{ f(x_{i}^{\ast})} : n \in \mathbb{N}, \text{ } x_1^{\ast}, \ldots, x_n^{\ast} \in \mathbb{L}^{\ast}, \text{ }\sup_{x \in \mathbb{L}} \sum_{i=1}^n \abs{x_i^{\ast}(x)} \leq 1 }.$$ We will denote this norm by $\norm{\cdot}$ or $\norm{\cdot}_{FBL\langle \mathbb{L} \rangle}$, indistinctly.

\begin{thm}[\mbox{\cite[Theorem 1.2]{ARA18}}]
	Consider $F$ to be the closure of $\lat \lbrace{ \delta_x : x \in \mathbb{L} \rbrace}$ under the norm $\|\cdot\|$ inside the Banach lattice of all functions $f \in \mathbb{R}^{\mathbb{L}^{\ast}}$ with $\|f\|<\infty$, endowed with the norm $\|\cdot\|$, the pointwise order and the pointwise operations.
	Then $F$, together with the assignment $\delta_{\mathbb{L}}(x)=\delta_x$, is the free Banach lattice generated by $\mathbb{L}$.
\end{thm}

\bigskip

Let us consider a lattice $\mathbb{M}$ and a sublattice $\mathbb{L}$ of $\mathbb{M}$, and let $i\colon \mathbb{L}\longrightarrow \mathbb{M}$ be the inclusion mapping. Note that the mapping $\delta_\mathbb{M}\circ i\colon \mathbb{L}\longrightarrow FBL\langle \mathbb{M} \rangle$ defines a bounded lattice homomorphism from $\mathbb{L}$ into the Banach lattice $FBL\langle \mathbb{M} \rangle$ and, consequently, the universal property of free Banach lattices described in Definition \ref{FBLGBL} yields a Banach lattice homomorphism $\hat \imath\colon FBL \langle \mathbb{L} \rangle \longrightarrow FBL \langle \mathbb{M}\rangle$. Note that, formally speaking, $\hat \imath(\delta_x)=\delta_{i(x)}$ for every $x\in \mathbb{L}$. Throughout the text, when we say that $FBL \langle \mathbb{L} \rangle$ is a sublattice of $FBL\langle \mathbb{M} \rangle$ we will mean that this mapping $\hat \imath$ is an isometric embedding (we will only explicitly mention this mapping in Section \ref{SectionIsomorphic} for the reader convenience).

As we have pointed out in the introduction, given a Banach space $E$ and a subspace $F$ of $E$, it is known that the  map $\hat \imath \colon  FBL[F] \longrightarrow FBL[E]$ induced by the inclusion of $F$ in $E$ is an isometric embedding when $F$ is $1$-complemented in $E$ \cite[Corollary 2.8]{ART18} and when $E=F^{**}$ \cite[Lemma 2.4]{dmrr}. Recently, T.~Oikberg, M.~Taylor, P.~Tradacete and V.~Troitsky\footnote{Result announced in a talk given by M.~Taylor on October 16th, 2020, at the \emph{Banach spaces webinars. 
}} characterized when $\hat \imath$ is an isometry in the following terms.

\begin{thm}[\mbox{T.~Oikberg, M.~Taylor, P.~Tradacete and V.~Troitsky}]\label{theo:peter}
Let $E$ be a Banach space and $F$ be a closed subspace of $E$. Let $i\colon F\longrightarrow E$ be the inclusion operator. The following are equivalent:
\begin{enumerate}
\item $\hat \imath\colon FBL[F]\longrightarrow FBL[E]$ is an isometry from $FBL[F]$ onto its range.
\item For every $n\in\mathbb N$ and every bounded operator $T\colon F\longrightarrow \ell_1^n$ there exists an operator $\hat T\colon E\longrightarrow \ell_1^n$ so that $\hat T\vert_{F}=T$ and $\Vert \hat T\Vert=\Vert T\Vert$.
\end{enumerate}
\end{thm}

In the next theorem we obtain a version for free Banach lattices over lattices of the implication (2)$\Rightarrow$(1) . Let $\mathbb L$ be a lattice. Given a lattice homomorphism $\phi \colon\mathbb L\longrightarrow \ell_1^n$, for some $n\in\mathbb N$, then
$$\phi(x)=\sum_{k=1}^n x_k^*(x)e_k$$
for some $x_k^* \colon\mathbb L\longrightarrow \mathbb R$. Moreover, since the  natural projections $P_k\colon\ell_1^n\longrightarrow \mathbb R$ are lattice homomorphisms, it follows that each $x_k^*=P_k \circ \phi$ is a lattice homomorphism for every $1\leq k\leq n$.

\begin{thm}\label{theo:condipedro}
Let $\mathbb L \subseteq \mathbb M$ be two lattices. Assume that for every $n \in \mathbb{N}$ and for every norm-bounded lattice homomorphism $\phi\colon\mathbb L\longrightarrow \ell_1^n$ there exists a norm preserving extension $\Phi\colon\mathbb M\longrightarrow \ell_1^n$. Then $FBL\langle\mathbb L\rangle$ is a sublattice of $FBL\langle\mathbb M\rangle$.
\end{thm}

\begin{proof}
Pick $f\in \lat\{\delta_x:x\in \mathbb L\}$, where $\lat\{\delta_x:x\in \mathbb L\}$ denotes indistinctly the vector lattice generated by $\{\delta_x:x\in \mathbb L\}$ inside $FBL\langle\mathbb L\rangle$ or $FBL\langle\mathbb M\rangle$, and let us prove that $\Vert f\Vert_{FBL\langle\mathbb M\rangle}=\Vert f\Vert_{FBL\langle\mathbb L\rangle}$. Let us prove $\Vert f\Vert_{FBL\langle\mathbb M\rangle} \geq \Vert f\Vert_{FBL\langle\mathbb L\rangle}$, which is the non-trivial inequality (the other inequality is a direct consequence of the fact that the map $\hat \imath$ induced by the inclusion has norm $\|\hat \imath\|=1$). To this end, pick $\varepsilon>0$ and choose $x_1^*,\ldots, x_n^*\in \mathbb{L}^*$ so that $\sum_{k=1}^n \vert f(x_k^*)\vert>\Vert f\Vert_{FBL\langle\mathbb L\rangle}-\varepsilon$ and $\sup\limits_{x\in\mathbb L} \sum_{k=1}^n \vert x_k^*(x)\vert\leq 1$. Define the lattice homomorphism $\phi\colon\mathbb L\longrightarrow \ell_1^n$ by the equation
$$\phi(x):=\sum_{k=1}^n x_k^*(x)e_k.$$
Note that, by the definition of the norm of a lattice homomorphism, we get that $\Vert \phi\Vert\leq 1$. By assumption there exists an extension $\Phi\colon\mathbb M\longrightarrow \ell_1^n$ with $\Vert \Phi\Vert\leq 1$. Now, $\Phi$ is given by $\Phi(x):=\sum_{k=1}^n \hat x_k^*(x)e_k$, where $\sum_{k=1}^n \vert \hat x_k^*(x)\vert\leq 1$ for every $x\in \mathbb M$ and $\hat x_k^*(x)=x_k^*(x)$ holds for every $x\in \mathbb L$. Since $f\in \lat \{\delta_x:x\in \mathbb L\}$ we get that $f(x^*)=f(y^*)$ whenever $x^*(x)=y^*(x)$ for every $x\in \mathbb L$. In particular, $f(\hat x_k^*)=f(x_k^*)$ holds for every $k\leq n$. Finally
$$\Vert f\Vert_{FBL\langle\mathbb L\rangle}-\varepsilon<\sum_{k=1}^n \vert f(x_k^*)\vert=\sum_{k=1}^n \vert f(\hat x_k^*)\vert\leq \Vert f\Vert_{FBL\langle \mathbb M\rangle}.$$
Since $\varepsilon>0$ was arbitrary we conclude that $\Vert f\Vert_{FBL\langle\mathbb M\rangle}=\Vert f\Vert_{FBL\langle\mathbb L\rangle}$ for every $f\in \lat\{\delta_x:x\in \mathbb L\}$. The conclusion follows from the fact that $FBL\langle \mathbb{L} \rangle = \overline{ \lat\{\delta_x:x\in \mathbb L\}}^{\Vert \cdot \Vert_{FBL\langle\mathbb L\rangle}}$.
\end{proof}

\begin{prob}
\label{Prob1}
We do not know whether the converse of Theorem \ref{theo:condipedro} is true, i.e.~whether if $\hat \imath\colon FBL\langle\mathbb L\rangle\longrightarrow FBL\langle\mathbb M\rangle$ is an isometry onto its range then 
for every $n\in\mathbb N$ and every norm-bounded lattice homomorphism $\phi \colon \mathbb{L} \longrightarrow \ell_1^n$ there exists a norm preserving extension $\Phi \colon \mathbb{M} \longrightarrow \ell_1^n$.
\end{prob}

\section{Isometric embeddings and local complementation}\label{section:main}

Let $\mathbb M$ be a lattice and $\mathbb L$ be a sublattice of $\mathbb M$. In Theorem \ref{theo:condipedro} we have obtained a sufficient condition which guarantees that $FBL\langle\mathbb L\rangle$ is an isometric sublattice of $FBL\langle\mathbb M\rangle$. This condition is, however, a bit difficult to check in practice because it involves the possibility of extending lattice homomorphisms taking values in $\ell_1^n$. In this section we provide a sufficient condition which is more tractable in the sense that it can be checked  just looking at the lattices $\mathbb L$ and $\mathbb M$. We provide several examples of lattices satisfying this condition and, in particular, one of this example shows that $FBL\langle \mathbb{L} \rangle $ is an AM-space for every lattice $\mathbb{L}$.

In the context of free Banach lattices over Banach spaces, if $F$ is a 1-complemented subspace of $E$ then $FBL[F]$ is a $1$-complemented sublattice of $FBL[E]$ \cite[Corollary 2.8]{ART18}. A lattice version of this result is given by the following proposition. For any set $A$, we write $1_A$ to denote the identity function $1_A \colon A \longrightarrow A$.

\begin{prop}\label{PropLatticeComplemented}
Let $\mathbb{L} \subseteq \mathbb{M}$ be two lattices. Let $i\colon \mathbb L\longrightarrow \mathbb M$ be the canonical inclusion and assume that there exists a lattice homomorphism $r\colon \mathbb M\longrightarrow\mathbb L$ so that $r\circ i=1_\mathbb L$ (in short, $\mathbb L$ is complemented in $\mathbb M$), then $FBL\langle \mathbb{L} \rangle$ is a 1-complemented Banach sublattice of $FBL\langle \mathbb{M} \rangle$.
\end{prop}

\begin{proof}
Let $i \colon \mathbb{L} \longrightarrow \mathbb{M}$ be the inclusion lattice homomorphism and let $r \colon \mathbb{M} \longrightarrow \mathbb{L}$ be a retraction. By using the universal property of the free Banach lattice generated by a lattice, we can find Banach lattice homomorphisms $\hat{\imath} \colon FBL \langle \mathbb{L} \rangle \longrightarrow FBL \langle \mathbb{M} \rangle$ and $\hat{r} \colon FBL \langle \mathbb{M} \rangle \longrightarrow FBL \langle \mathbb{L} \rangle$ such that $\| \hat{\imath} \| = \| \hat{r} \| = 1$ and
 $(\hat{r} \circ \hat{\imath}) \circ \delta_\mathbb{L}=\hat{r} \circ (\hat{\imath} \circ \delta_\mathbb{L})=\hat{r} \circ (\delta_{\mathbb{M}} \circ i)= \delta_{\mathbb{L}} \circ r \circ i = \delta_{\mathbb{L}}$. It follows from the uniqueness in the universal property that $\hat{r} \circ \hat{\imath}=
  1_{FBL\langle \mathbb{L} \rangle}$, so we are done. 
\end{proof}

The assumption of being complemented, though being an intrinsic condition on the lattices $\mathbb L$ and $\mathbb M$, is still quite restrictive. For instance, it is proved in \cite[Lemma 3.6]{jdthesis} that if $\mathbb M$ is a linearly ordered set then any subset $\mathbb L\subseteq \mathbb M$ satisfies that $FBL\langle\mathbb L\rangle$ is an isometric sublattice of $FBL\langle\mathbb M\rangle$. However, it is easy to construct examples where there is no complementation condition (for instance, $\mathbb Q\subseteq \mathbb R$). Because of this reason we look for a weaker intrinsic criterion which still implies isometric containment. In order to do so, we look again at the case of Banach spaces, and we look for a version of the concept of locally complemented Banach spaces. Being inspired by the original definition given by Kalton \cite{kalton84}, we propose the following:

\begin{defn}\label{defi:lattlocomp}
Let $\mathbb L \subseteq \mathbb{M}$ be two lattices. We say that $\mathbb L$ is \textit{locally complemented} in $\mathbb M$ if for every finite sublattice $\mathbb F $ of $\mathbb M$ there exists a lattice homomorphism $T \colon\mathbb F \longrightarrow \mathbb L$ such that
$T(x)=x\ \text{ for every } x\in \mathbb F \cap \mathbb L.$
\end{defn}

It is clear that if $\mathbb L$ is complemented in $\mathbb M$ then it is locally complemented. Many examples exhibited in Proposition \ref{PropositionExamplesLocComp} reveals that local complementation does not imply complementation.

Before presenting examples of locally complemented lattices, we prove the following theorem which justifies our interest in this concept.

\begin{thm}\label{TheoremLocComExtDual}
Let $\mathbb L \subseteq \mathbb{M}$ be two lattices. If $\mathbb L$ is locally complemented in $\mathbb M$, then for any finite-dimensional Banach lattice $X$ and any lattice homomorphism $\phi\colon \mathbb L\longrightarrow X$ of finite norm there exists a lattice homomorphism $\Phi\colon \mathbb M\longrightarrow X$ which extends $\phi$ and so that $\Phi(\mathbb M)\subseteq \overline{\phi(\mathbb L)}$. In particular, $\Vert \Phi\Vert= \Vert\phi\Vert$.
\end{thm}

\begin{proof}

Pick any finite sublattice $\mathbb{F}$ of $\mathbb M$. By assumption there exists a lattice homomorphism $T_{\mathbb{F}}\colon \mathbb{F} \longrightarrow \mathbb L$ such that $T_{\mathbb{F}}(x)=x\ \text{ for every } x\in \mathbb{F}\cap \mathbb L.$
Now define $\phi_{\mathbb{F}}:=\phi\circ T_{\mathbb{F}}\colon \mathbb{F} \longrightarrow X$, which is a lattice homomorphism. Moreover, it is clear that, given $x\in \mathbb{F}$, then
$$\Vert \phi_{\mathbb{F}}(x)\Vert=\Vert \phi(T_{\mathbb{F}}(x))\Vert\leq \Vert \phi\Vert,$$
which, in other words, means that $\phi_{\mathbb{F}}(x)\in \Vert \phi\Vert B_{X}$. We extend $\phi_\mathbb{F}$ to $\Phi_{\mathbb{F}}\colon \mathbb M\longrightarrow \Vert \phi\Vert B_{X}$ by the equation
$$\Phi_{\mathbb{F}}(x):=\left\{\begin{array}{cc}
\phi_{\mathbb{F}}(x) & \text{if } x\in \mathbb{F},\\
0 & \text{if } x\notin \mathbb{F}.
\end{array} \right.$$
Notice that $\Phi_{\mathbb{F}}$ is not a lattice homomorphism. However, as it is defined, it is obvious that $\Phi_{\mathbb{F}}$ belongs to the compact space $(\Vert \phi\Vert B_{X})^\mathbb M$, endowed with the product topology. Define
$$\mathcal S:=\{\mathbb{F}\subseteq \mathbb M: \mathbb{F}\mbox{ is a finite sublattice}\}.$$
$\mathcal S$ is a directed set with the order $\mathbb{E}\leq \mathbb{F}$ if and only if  $\mathbb{E}\subseteq \mathbb{F}$. With this point of view, $(\Phi_{\mathbb{F}})_{\mathbb{F}\in \mathcal S}$ is a net in the compact space $((\Vert \phi\Vert B_{X},\Vert\cdot\Vert))^\mathbb M$. By compactness, we get a cluster point $\Phi$ of the net $(\Phi_{\mathbb{F}})_{\mathbb{F}\in \mathcal S}$. Let us prove that $\Phi$ satisfies the desired requirements. First, let us prove that $\Phi$ is a lattice homomorphism. To this end, pick $x,y\in \mathbb M$. Then, for any $\mathbb{F}\in\mathcal S$ such that $\{x,y\}\subseteq \mathbb{F}$ we get
$$\Phi_{\mathbb{F}}(x\vee y)=\phi_{\mathbb{F}}(x\vee y)=\phi_{\mathbb{F}}(x)\vee \phi_{\mathbb{F}}(y)=\Phi_{\mathbb{F}}(x)\vee\Phi_{\mathbb{F}}(y),$$
where we have used that $\phi_{\mathbb{F}}$ is a lattice homomorphism for every $\mathbb{F}$. Since $\Phi$ is a cluster point of $(\Phi_{\mathbb{F}})_{\mathbb{F}\in \mathcal S}$ and the lattice operations on $X$ are norm-continuous we get that 
$$\Phi(x\vee y)=\Phi(x)\vee \Phi(y).$$
A similar argument shows that $\Phi$ also preserves infima. The arbitrariness of $x,y$ implies that $\Phi$ is a lattice homomorphism.

Now we prove that $\Phi(x)=\phi(x)$ for every $x\in \mathbb L$. Pick $x\in \mathbb L$. For every $\mathbb{F}\in\mathcal S$ with $x\in \mathbb{F}$ we get that
$$\Phi_{\mathbb{F}}(x)=\phi_{\mathbb{F}}(x)=\phi(T_{\mathbb{F}}(x))=\phi(x),$$
since $T_{\mathbb{F}}(x)=x$ for every $x\in \mathbb{F}\cap \mathbb L$. This shows that $\Phi$ extends $\phi$.

Let us finally prove that $\Phi(\mathbb M)\subseteq \overline{\phi(\mathbb L)}$. To this end, pick any $x\in \mathbb M$ and notice that, for every $\mathbb{F} \in\mathcal S$ with $x\in \mathbb{F}$ we get that
$$\Phi_{\mathbb{F}}(x)=\phi_{\mathbb{F}}(x)=\Phi(T_{\mathbb{F}}(x))\in \phi(\mathbb L),$$
and now the cluster condition implies that $\Phi(x)\in\overline{\phi(\mathbb L)}$, which concludes the proof. 
\end{proof}

\begin{rem}
After an inspection of the proof, one might think that we can replace $X$ being finite-dimensional with $X$ being a dual Banach lattice by using a Lindenstrauss compactness argument involving the weak*-compactness of the unit ball of a dual Banach space. However, this technique does not work since the resulting mapping $\Phi$ might not be a lattice homomorphism because, in general, it is not true that the lattice operations in a dual Banach lattice are $w^*$-continuous. Indeed, if one considers the Rademacher sequence $(r_n)_{n \in \mathbb{N}}$ in $L_2([0,1])$, then $(r_n)_{n \in \mathbb{N}}\rightarrow 0$ weakly, but $r_n\vee (-r_n)$ is the constant function $1$ for every $n\in\mathbb N$. So $\lim_{n \rightarrow +\infty}( r_n\vee (-r_n))=1\neq 0= (\lim_{n \rightarrow +\infty} r_n)\vee (\lim_{n \rightarrow +\infty} -r_n)$, where these limits are taken with respect to the weak topology.
\end{rem}

An application of the previous theorem to $X=\ell_1^n$ together with Theorem \ref{theo:condipedro} yields the desired consequence.

\begin{cor}\label{cor:fbllatisubs}
Let $\mathbb{L} \subseteq \mathbb{M}$ be two lattices. If $\mathbb{L}$ is locally complemented in $\mathbb{M}$, then $FBL\langle \mathbb{L} \rangle$ is an isometric sublattice of $FBL\langle\mathbb M\rangle$.
\end{cor}

Let us devote the end of this section to provide relevant examples related to the content of Section \ref{section:main}. Up to our knownledge, there is not any concept related to locally complemented sublattices in the literature, so we begin the section by showing many examples of such kind of sublattices in a given lattice.

First, let us recall the following well-known concepts in lattice theory.

\begin{defn}
Let $\mathbb{L} \subseteq \mathbb{M}$ be two lattices.
\begin{enumerate}

\item We say that $\mathbb{L}$ is an \textit{ideal} in $\mathbb{M}$ if $x \in \mathbb{L}$ whenever there is $y \in \mathbb{L}$ with $x \leq y$.

\item We say that $\mathbb{L}$ is a \textit{filter} in $\mathbb{M}$ if $x \in \mathbb{L}$ whenever there is $y \in \mathbb{L}$ with $x \geq y$.

\end{enumerate}
\end{defn}

The notions of filter and ideal are of capital importance in Boolean algebras and play an important role in the Stone duality theorem of Boolean algebras (see \cite{monkhandb}).

Now we are able to exhibit several natural examples of locally complemented sublattices.

\begin{prop}\label{PropositionExamplesLocComp}
Let $\mathbb{L} \subseteq \mathbb{M}$ be two lattices. If,

\begin{enumerate}

\item\label{PropositionExamplesLocComp1} $\mathbb{M}$ is a linearly ordered set, or

\item $\mathbb{M}$ is a lattice and $\mathbb{L}$ is an ideal in $\mathbb{M}$, or

\item $\mathbb{M}$ is a lattice and $\mathbb{L}$ is a filter in $\mathbb{M}$, or

\item\label{ItemLatticeInsideBoundedLattice} $\mathbb{M} = \mathbb{L} \cup \{m,M\}$ with the property that $m=\min \mathbb M$ and $M=\max \mathbb M$,

\end{enumerate}
then $\mathbb{L}$ is locally complemented in $\mathbb{M}$.\end{prop}

\begin{proof}

Let $\mathbb{F} \subseteq \mathbb{M}$ be a finite sublattice. If $\mathbb{F} \cap \mathbb{L} = \emptyset$, we can take as $T \colon \mathbb{F} \longrightarrow \mathbb{L}$ any constant map. Thus, without loss of generality, we suppose $\mathbb{L}\cap\mathbb{F} \neq \emptyset$. 

For (1), since $\mathbb{M}$ is linearly ordered, the map $T \colon \mathbb{F} \longrightarrow \mathbb{L}$ given by $$T(x)= \left\{ \begin{array}{lcc}
             \inf\{ y \in \mathbb{L}\cap\mathbb{F} : y \geq x\} & \text{if there exists }y \in \mathbb{L}\cap\mathbb{F} \text{ with }y \geq x, \\
             \\ \sup\{ y \in \mathbb{L}\cap\mathbb{F} : y \leq x\} & \text{otherwise,}
             \end{array}
   \right.$$is a well-defined lattice homomorphism such that $T(x) = x$ for every $x \in \mathbb{L}\cap\mathbb{F}$.

For (2), set $z:= \sup (\mathbb{F} \cap \mathbb{L}$). The fact that $\mathbb{L}$ is an ideal in $\mathbb{M}$ allows us to define a map $T \colon \mathbb{F} \longrightarrow \mathbb{L}$ given by $T(x) = x \wedge z$ for every $x \in \mathbb{F}$. The fact that $\mathbb{M}$ is distributive guarantees that $T$ is a lattice homomorphism. Now, if $x \in \mathbb{F} \cap \mathbb{L}$ we clearly have that $T(x) = x \wedge z = x$.

The proof of (3) is similar to the previous one, taking $z:= \min (\mathbb{F} \cap \mathbb{L})$ and $T \colon \mathbb{F} \longrightarrow \mathbb{L}$ the lattice homomorphism given by $T(x) = x \vee z$ for every $x \in \mathbb{F}$.

Finally, for (4), we define $T\colon \mathbb{F} \longrightarrow \mathbb{L}$ by $T(x)=x$ for every $x\in \mathbb{F}\cap \mathbb L$, $T(m)=\inf (\mathbb{F}\cap \mathbb L)$ if $m \in \mathbb{F}$, and $T(M)=\sup (\mathbb{F}\cap \mathbb L)$ if $M \in \mathbb{F}$.
\end{proof}

As we have said before, if $\mathbb L$ is a complemented sublattice of $\mathbb M$, then it is locally complemented. The converse is not true. For example, $\mathbb Q$ is not complemented in $\mathbb R$ but it is locally complemented by \eqref{PropositionExamplesLocComp1} in Proposition \ref{PropositionExamplesLocComp}. However, a kind of converse can be established when we deal with finite sublattices in the following sense.

\begin{prop}
Let $\mathbb{L} \subseteq \mathbb{M}$ be two lattices. If $\mathbb{L}$ is finite and locally complemented in $\mathbb M$, then $\mathbb L$ is complemented in $\mathbb M$.
\end{prop}

\begin{proof}
Since $\mathbb L$ is distributive, there exists a Boolean algebra $\mathbb B$ and an injective lattice homomorphism $\psi\colon \mathbb L\longrightarrow \mathbb B$ \cite[Theorem II.19]{Gratzer}. Since $\mathbb{L}$ is finite, we can assume that $\mathbb{B}=\{0,1\}^n$ for certain $n\in\mathbb N$. Now we consider $\varphi\colon \mathbb B\longrightarrow \ell_1^n$ by $\varphi((x_1,\ldots, x_n))=(y_1,\ldots, y_n)\in \ell_1^n$, where $y_i=\frac{1}{n}$ if $x_i=1$ and $y_i=-\frac{1}{n}$ if $x_i=0$. This defines an injective lattice homomorphism $\phi= \varphi \circ \psi\colon \mathbb L\longrightarrow \ell_1^n$. By Theorem \ref{TheoremLocComExtDual} we can find an extension $\Phi \colon \mathbb M\longrightarrow \ell_1^n$ so that $\Phi (\mathbb M)\subseteq \overline{\phi(\mathbb L)}=\phi(\mathbb L)$, where the last equality holds since $\mathbb L$ is finite. If we consider a lattice homomorphism $q\colon \phi(\mathbb L)\longrightarrow \mathbb L$ so that $\phi \circ q=id_{\phi(\mathbb{L})}$ and $q\circ \phi=id_{\mathbb L}$, then $r:=q\circ \Phi\colon \mathbb M\longrightarrow \mathbb L$ defines a mapping so that $r\circ i=id_{\mathbb L}$, where $i\colon \mathbb L\longrightarrow \mathbb M$ denotes the inclusion lattice homomorphism. Hence $\mathbb L$ is complemented in $\mathbb M$, as desired.
\end{proof}

At first glance example (4) in Proposition \ref{PropositionExamplesLocComp} might seem naïve. Nevertheless, on one hand, it is not true that if $\mathbb{M}=\mathbb{L} \cup \mathbb{F} $ with $\mathbb{F}$ being a finite sublattice then $\mathbb{L}$ is locally complemented in $\mathbb{M}$; a simple counterexample to this fact can be seen in Example \ref{examp:nosublattice}. On the other hand, for every lattice $\mathbb{L}$ we can consider the lattice $\mathbb{M}:=\mathbb{L} \cup \{m,M\}$ obtained adding to $\mathbb{L}$ a minimum $m$ and a maximum $M$. Thus, a combination of example (4) in Proposition \ref{PropositionExamplesLocComp}, Corollary \ref{cor:fbllatisubs} and \cite[Theorem 2.7]{amrr2020} yields that $FBL\langle \mathbb{L} \rangle $ is isomorphic to a sublattice of a $C(K)$-space for every lattice $\mathbb{L}$. Recall that a Banach lattice $X$ is said to be an \textit{AM-space} if $\|x \vee y \|= \max\{\|x\|,\|y\|\}$ for every positive disjoint elements $x,y\in X$. By the classical Kakutani-Bohnenblust-Krein Theorem (see, for instance, \cite[Theorem 3.6]{AA}), a Banach lattice is an AM-space if and only if it is lattice isometric to a sublattice of a $C(K)$-space. Thus, we have proved the following result.
\begin{thm}
\label{TheoAMspace}
$FBL\langle \mathbb{L} \rangle $ is isomorphic to an AM-space for every lattice $\mathbb{L}$.
\end{thm}

We have finished the previous section with Problem \ref{Prob1}, asking whether the converse of Theorem \ref{theo:condipedro} is true.
We have shown that if $\mathbb L$ is a locally complemented sublattice of $\mathbb M$ then $\mathbb L$ and $\mathbb M$ satisfy the hypothesis of Theorem \ref{theo:condipedro}. Nevertheless, we do not know whether local complementation is equivalent to the hypothesis of Theorem \ref{theo:condipedro}:
\begin{prob}
	Let $\mathbb M$ be a lattice and $\mathbb L$ be a sublattice of $\mathbb{M}$. Assume that $\mathbb L$ and $\mathbb M$ satisfy the hypothesis of Theorem \ref{theo:condipedro}. Is $\mathbb L$ locally complemented in $\mathbb M$?
\end{prob}

\section{Isomorphic  embeddings}\label{SectionIsomorphic}

In this section we deal with isomorphic lattice embeddings. Namely, we provide necessary and sufficient conditions for the map $\hat{\imath} \colon FBL \langle \mathbb{L} \rangle \longrightarrow FBL \langle \mathbb{M} \rangle$ to be an isomorphic lattice embedding. In order to do so, let us assume that $\mathbb L$ is a lattice with maximum $M$ and minimum $m$. Then, $K_{\mathbb L}:=\{x^*\in \mathbb L^*: \max\{\vert x^*(m)\vert, \vert x^*(M)\vert\}=1\} \subseteq [-1,1]^\mathbb{L}$ is a compact space when endowed with the product topology. Consider the map $\phi_\mathbb L \colon FBL\langle \mathbb L\rangle\longrightarrow C(K_\mathbb L)$ defined by the equation
$$\phi_\mathbb L(f)(x^*):=f(x^*)\ \mbox{ for every }  x^*\in K_\mathbb L.$$
It is immediate that $\phi_\mathbb L$ is injective and that $\|\phi_\mathbb L\| \leq 1$ by the definition of the norm in $FBL\langle \mathbb{L} \rangle $. Indeed, it was proved in \cite[Theorem 2.7]{amrr2020} that $\phi_\mathbb L$ is surjective and that $\frac{1}{2}\|f\| \leq \|\phi_\mathbb L (f)\| \leq \|f\|$ for every $f\in FBL\langle \mathbb{L} \rangle $, so $\phi_{\mathbb{L}}$ is a  lattice isomorphism and $FBL\langle \mathbb{L} \rangle $ is $2$-lattice isomorphic to $C(K_\mathbb L)$.

Now, given a lattice $\mathbb M$ with maximum and minimum $M$ and $m$ respectively, a sublattice $\mathbb L\subseteq \mathbb M$ so that $m,M\in\mathbb L$ and the inclusion $i\colon \mathbb L\longrightarrow \mathbb M$, we can consider the Banach lattice isomorphisms $\phi_{\mathbb{L}} \colon FBL\langle \mathbb L\rangle\longrightarrow C(K_\mathbb L)$ and $\phi_M \colon FBL\langle \mathbb M\rangle\longrightarrow C(K_\mathbb M)$. Take $r\colon K_\mathbb M\longrightarrow K_\mathbb L$ the restriction operator, that is, $r(x^*):=x^*\vert_\mathbb L$ for every $x^* \in K_\mathbb M$, and consider the composition operator $C_r \colon C(K_\mathbb L)\longrightarrow C(K_\mathbb M)$ given by the equation $C_r(g):=g\circ r$ for every $g\in C(K_\mathbb L)$. It follows from the definition of $C_r$, $\phi_\mathbb L$ and $\phi_\mathbb M$ that $C_r\circ \phi_\mathbb L=\phi_\mathbb M\circ \hat \imath$ or, in other words,
$$\hat \imath=\phi_\mathbb M^{-1}\circ C_r\circ \phi_\mathbb L.$$
Now we get that $\hat \imath\colon FBL\langle\mathbb L\rangle\longrightarrow FBL\langle\mathbb M\rangle$ is an (into) isomorphism if and only if  $C_r$ is an into isomorphism. It is clear that $C_r$ is an into isomorphism if and only if  $r$ is surjective (in such a case it is obvious that $C_r$ is even an isometry), which is in turn equivalent to the fact that $C_r$ is injective  (see, for instance, \cite[Corollary 4.2.3 and Proposition 7.7.2]{Semadeni}). Now $r$ is surjective if and only if  every lattice homomorphism $y^*\colon \mathbb L\longrightarrow [-1,1]$ with $\max\{\vert y^*(m)\vert, \vert y^*(M)\vert\}=1$ admits an extension to a lattice homomorphism $\hat y^*\colon \mathbb M\longrightarrow [-1,1]$. Notice that, up to a suitable translation, the previous condition is equivalent to the fact that every lattice homomorphism $y^* \colon \mathbb L\longrightarrow [-1,1]$ admits an extension to a lattice homomorphism $\hat y^*\colon \mathbb M\longrightarrow [-1,1]$.  Thus, we have obtained the following result. 

\begin{prop}\label{prop:latisomaxminambos}
Let $\mathbb M$ be a lattice with maximum and minimum $M$ and $m$ respectively. Let $\mathbb L \subseteq \mathbb{M}$ be a sublattice containing $m$ and $M$. The following are equivalent:
\begin{enumerate}
\item $\hat \imath\colon FBL\langle\mathbb L\rangle\longrightarrow FBL\langle\mathbb M\rangle$ is an into isomorphism; 

\item $\hat \imath \colon FBL\langle\mathbb L\rangle\longrightarrow FBL\langle\mathbb M\rangle$ is injective;

\item Every lattice homomorphism $y^* \colon \mathbb L\longrightarrow [-1,1]$ admits an extension to a lattice homomorphism $\hat y^*\colon \mathbb M\longrightarrow [-1,1]$.
\end{enumerate}
\end{prop}

Our aim is now to remove the assumptions on the existence of a maximum and a minimum in $\mathbb M$ and $\mathbb L$. This will be done in two steps, where we will apply the results of the previous section. First, we will remove the assumption on $\mathbb L$.

\begin{lem}\label{lemma:latisomaxgran}
Let $\mathbb M$ be a lattice with maximum and minimum $M$ and $m$ respectively. Let $\mathbb L$ be a sublattice of $M$ and $i\colon \mathbb L\longrightarrow \mathbb M$ be the inclusion operator. The following are equivalent: 
\begin{enumerate}
\item $\hat \imath \colon FBL\langle\mathbb L\rangle\longrightarrow FBL\langle\mathbb M\rangle$ is an into isomorphism; 

\item $\hat \imath \colon FBL\langle\mathbb L\rangle\longrightarrow FBL\langle\mathbb M\rangle$ is injective;

\item Every lattice homomorphism $y^* \colon \mathbb L\longrightarrow [-1,1]$ admits an extension to a lattice homomorphism $\hat y^* \colon \mathbb M\longrightarrow [-1,1]$.
\end{enumerate}
\end{lem}

\begin{proof}
Let $j\colon \mathbb L\longrightarrow \mathbb L\cup\{m,M\}$ and $k\colon \mathbb L\cup\{m,M\}\longrightarrow \mathbb M$ be the canonical inclusions. It is clear that $i=k\circ j$, from where $\hat \imath=\hat k\circ \hat j$, where $\hat j\colon FBL\langle\mathbb L\rangle\longrightarrow FBL\langle\mathbb L\cup\{m,M\}\rangle$ and $\hat k\colon FBL\langle\mathbb L\cup\{m,M\}\rangle\longrightarrow FBL\langle\mathbb M\rangle$ are maps induced by $j$ and $k$ respectively. By (4) in Proposition \ref{PropositionExamplesLocComp} we get that $\mathbb L$ is locally complemented in $\mathbb L\cup\{m,M\}$ and, consequently, $\hat j$ is an isometry by Corollary \ref{cor:fbllatisubs}. Hence, $\hat \imath$ is an isomorphism (resp.~injective) if and only if  $\hat k$ is an isomorphism (resp. injective). By Proposition \ref{prop:latisomaxminambos} we get that this is equivalent to the fact that every element of $(\mathbb L\cup\{m,M\})^*$ extends to an element of $\mathbb M^*$. Taking $X=\R$ in Theorem \ref{TheoremLocComExtDual} we get that this condition is in turn equivalent to the fact that every element of $\mathbb L^*$ extends to an element of $\mathbb M^*$, which concludes the proof.
\end{proof}

The previous lemma will allow us to obtain a complete characterization of when $\hat \imath$ is an into isomorphism.

\begin{thm}\label{theo:caraisomorpinto}
Let $\mathbb M$ be a lattice and $\mathbb L\subseteq \mathbb{M}$ be a sublattice. The following are equivalent: 
\begin{enumerate}
\item $\hat \imath \colon FBL\langle\mathbb L\rangle\longrightarrow FBL\langle\mathbb M\rangle$ is an into isomorphism 

\item $\hat \imath \colon FBL\langle\mathbb L\rangle\longrightarrow FBL\langle\mathbb M\rangle$ is injective.

\item Every lattice homomorphism $y^* \colon \mathbb L\longrightarrow [-1,1]$ admits an extension to a lattice homomorphism $\hat y^* \colon \mathbb M\longrightarrow [-1,1]$.
\end{enumerate}
\end{thm}

\begin{proof}
Set $\mathbb O:=\mathbb M\cup\{m,M\}$ the lattice obtained adding a maximum $M$ and a minimum $m$ to $\mathbb{M}$. We can consider $j\colon \mathbb M\longrightarrow \mathbb O$ the canonical inclusion and define $k:=j\circ i \colon \mathbb L\longrightarrow \mathbb O$. We get that $\hat k=\hat j\circ \hat \imath$, where $\hat j \colon FBL\langle\mathbb M\rangle\longrightarrow FBL\langle\mathbb O\rangle$ and $\hat k \colon FBL\langle\mathbb L\rangle\longrightarrow FBL\langle\mathbb O\rangle$ are the corresponding induced operators. By (4) in Proposition \ref{PropositionExamplesLocComp} we get that $\mathbb M$ is locally complemented in $\mathbb O$ and, consequently, $\hat j$ is an isometry by Corollary \ref{cor:fbllatisubs}. Now it is clear that $\hat \imath$ is an into isomorphism (respectively injective) if and only if  so is $\hat k$ and, by Lemma \ref{lemma:latisomaxgran}, this is equivalent to the fact that every element in $\mathbb L^*$ extends to an element in $\mathbb O^*$, which is in turn equivalent to the fact that every element in $\mathbb L^*$ extends to an element in $\mathbb M^*$ by Theorem \ref{TheoremLocComExtDual}, as desired.
\end{proof}

As an application we can easily obtain examples of lattices $\mathbb{L} \subseteq \mathbb M$ for which the canonical inclusion is not an isomorphism.

\begin{example}\label{examp:nosublattice}
Let $\mathbb{M} = \{ m,a,b,M \}$ be the lattice with four elements with $m$ being the minimum, $M$ the maximum, and $a$ and $b$ not comparable between them. Then, every lattice homomorphism $x^* \in \mathbb{M}^*$ satisfies $x^*(a)\vee x^*(b)=x^*(M)$ and  $x^*(a)\wedge x^*(b)=x^*(m)$, so $x^*(a),x^*(b) \in \{x^*(m),x^*(M)\}$ and $x^*$ takes at most two different values.
Nevertheless, in the sublattice $\mathbb{L} = \{ m,a,M \}$, which is linearly ordered, we can easily construct lattice homomorphisms taking three different values. Such homomorphisms in $\mathbb{L}^*$ cannot be extended to homomorphisms in $\mathbb{M}^*$. Thus, by Theorem \ref{theo:caraisomorpinto}, $\hat \imath \colon FBL\langle\mathbb L\rangle\longrightarrow FBL\langle\mathbb M\rangle$ is not an into isomorphism.
\end{example}

\section{An isomorphic embedding which is not an isometry}\label{section:isomonotisome}

In this section we aim to give an example of a lattice $\mathbb M$ and a sublattice $\mathbb L$ so that the mapping $\hat \imath \colon FBL\langle\mathbb L\rangle\longrightarrow FBL\langle\mathbb M\rangle$ is an into isomorphism but not an isometry.

Let $\mathbb{M} = \{1,2,3\} \times \{1,2,3\}$ endowed with the coordinatewise order, and take the sublattice $\mathbb{L} = \{(1,1),(2,2),(2,3),(3,2),(3,3)\}$. 

\begin{figure}[h]
	\begin{subfigure}[b]{.45\textwidth}
		\centering
		\begin{tikzpicture}\small
		\def\x{1.5};
		\def\y{1};
		\draw(2*\x,1*\y) node(n11){$(1,1)$};
		\draw(1.5*\x,2*\y) node(n12){$(1,2)$};
		\draw(2.5*\x,2*\y) node(n21){$(2,1)$};
		\draw(1*\x,3*\y) node(n13){$(1,3)$};
		\draw(2*\x,3*\y) node(n22){$(2,2)$};
		\draw(3*\x,3*\y) node(n31){$(3,1)$};
		\draw(1.5*\x,4*\y) node(n23){$(2,3)$};
		\draw(2.5*\x,4*\y) node(n32){$(3,2)$};
		\draw(2*\x,5*\y) node(n33){$(3,3)$};

		\draw(n11) -- (n21);
		\draw(n11) -- (n12);
		\draw(n21) -- (n22);
		\draw(n12) -- (n22);
		\draw(n12) -- (n13);
		\draw(n21) -- (n31);
		\draw(n31) -- (n32);
		\draw(n13) -- (n23);
		\draw(n32) -- (n33);
		\draw(n22) -- (n32);
		\draw(n22) -- (n23);
		\draw(n23) -- (n33);
		\end{tikzpicture}
		\caption{Representation of $\mathbb{M}$.}
		\label{fig:cu_lattice}
	\end{subfigure}
	\hspace*{.5cm}
	\begin{subfigure}[b]{.45\textwidth}
		\centering
		\begin{tikzpicture}\small
		\def\x{1.5};
		\def\y{1};
		\draw(2*\x,1*\y) node(n11){$(1,1)$};
		\draw(2*\x,2.5*\y) node(n22){$(2,2)$};
		\draw(1.5*\x,3.5*\y) node(n23){$(2,3)$};
		\draw(2.5*\x,3.5*\y) node(n32){$(3,2)$};
		\draw(2*\x,4.5*\y) node(n33){$(3,3)$};

		\draw(n11) -- (n22);
		\draw(n32) -- (n33);
		\draw(n22) -- (n32);
		\draw(n22) -- (n23);
		\draw(n23) -- (n33);
		\end{tikzpicture}
		\caption{Representation of the sublattice $\mathbb{L}$.}
		\label{fig:cu_lattice_shards}
	\end{subfigure}
\end{figure}

\begin{prop} The map $\hat{\imath} \colon FBL \langle \mathbb{L} \rangle \longrightarrow FBL\langle \mathbb{M} \rangle$ is an isomorphic lattice embedding.\end{prop}

\begin{proof}
Let $x^* \colon \mathbb{L} \longrightarrow [-1,1]$ be a lattice homomorphism. Note that $$x^*((2,2)) = \min \{ x^*((2,3)), x^*((3,2))\}$$and $$x^*((3,3)) = \max \{ x^*((2,3)), x^*((3,2))\}.$$ Let us define an extension $y^* \colon \mathbb{M} \longrightarrow [-1,1]$ of $x^*$. We put $y^*(x) := x^*(x)$ for every $x \in \mathbb{L}$. We have to define $y^*((1,2)), y^*((1,3)), y^*((2,1))$ and $y^*((3,1))$ so that $y^*$ is a lattice homomorphism. Notice that a necessary and sufficient condition for $y^*$ to be a lattice homomorphism is that the restriction of $y^*$ to any diamond of $\mathbb{M}$ (i.e.~its restriction to any sublattice isomorphic to the lattice defined in Example \ref{examp:nosublattice}) is a lattice homomorphism; in particular, the restriction of $y^*$ to any diamond must take at most two different values.

We must distinguish several cases:

\begin{itemize}

\item If $x^*((2,3)) = x^*((3,2))$, we define $y^*(y) := x^*((2,3))$ for every $y \in \{(2,1),(3,1)\}$ and $y^*(y):=x^*((1,1))$ if $y\in \{(1,2),(1,3)\}$.

\item If $x^*((3,2)) < x^*((2,3))$, then necessarily $y^*((1,3))=x^*((2,3))$ and $y^*((1,2))=x^*((2,2))$. Since $x^*((1,1))\leq x^*((2,2))$, this implies that $y^*((2,1))=x^*((1,1))$ and so $y^*((3,1))=y^*((2,1))=x^*((1,1))$.

\item The case that $x^*((2,3)) < x^*((3,2))$ is symmetric to the previous one.
\end{itemize}
In any of the previous cases a standard computation shows that the map $y^*$ is a lattice homomorphism. Thus, every lattice homomorphism in $\mathbb{L}^*$ extends to a lattice homomorphism in $\mathbb{M}^*$, so the conclusion follows from Theorem \ref{theo:caraisomorpinto}.
\end{proof}

\bigskip

From the above estimates it is clear that every lattice homomorphism $x^*\colon \mathbb L\longrightarrow [-1,1]$ takes at most three different values (either on the chain $\{(1,1),(2,2),(2,3)\}$ or on the chain $\{(1,1),(2,2),(3,2)\}$) and can always be extended to $\mathbb{M}$. Furthermore, we have shown that if it takes three different values then it extends to $\mathbb M$ in a unique way.
The following result shows that $\mathbb L$ and $\mathbb M$ satify our purposes.

\begin{prop}\label{PropExample2}
The isomorphic lattice embedding $\hat{\imath} \colon FBL \langle \mathbb{L} \rangle \longrightarrow FBL\langle \mathbb{M} \rangle$ is not isometric.
\end{prop}

\begin{proof}
We will prove that there exists some $f \in FBL\langle \mathbb{L} \rangle$ such that $\|f\|_{FBL \langle \mathbb{L} \rangle} > \| \hat{\imath}(f) \|_{FBL\langle \mathbb{M} \rangle}$. We use the same notation than in Section \ref{SectionIsomorphic}.
Recall that $\phi_\mathbb L \colon FBL\langle \mathbb L\rangle\longrightarrow C(K_\mathbb L)$ is a lattice isomorphism, so any continuous function $g\colon K_\mathbb{L} \longrightarrow \mathbb{R}$ is identified through $\phi_\mathbb{L}$ with an element $\phi_\mathbb L^{-1}(g)\in FBL\langle \mathbb{L} \rangle$, which is just obtained extending $g$ to $\mathbb{L}^*$ through the equality $g(\lambda x^*):=\lambda g(x^*)$ for every $x^* \in K_\mathbb{L}$ and every $\lambda \in [0,1]$.

Fix $0<\varepsilon<\frac{1}{2}$ and set the lattice homomorphisms $x_1^*, x_2^* \in K_\mathbb{L}$ given by:

\begin{itemize}
	\item $x_1^*((1,1))= -1, \text{ } x_1^*((2,2))=x_1^*((2,3))= 0, \text{ } x_1^*((3,2)) = x_1^*((3,3)) =\varepsilon$.
	
	\item $x_2^*((1,1)) = 0, \text{ } x_2^*((2,2)) = x_2^*((3,2))= \varepsilon, \text{ } x_2^*((2,3)) =x_2^*((3,3))=1$.
\end{itemize}

For $i = 1, 2$ let $V_i := \{ x^* \in K_{\mathbb{L}} : |x^*(x) - x_i^*(x)| < \frac{\varepsilon}{2} \text{ for every } x \in \mathbb{L} \}$ be a neighborhood of $x_i^*$ in $K_{\mathbb{L}}$. By the classical Tietze's extension theorem, we can consider $g\colon K_\mathbb{L} \longrightarrow [0,1]$ a continuous function such that $g(x_1^*)=g(x_2^*)=1$ and $g(x^*)=0$ for every $x^* \in K_\mathbb{L} \setminus (V_1 \cup V_2)$. Let $f \in FBL\langle \mathbb L\rangle$ be the extension of $g$ to $\mathbb{L}^*$ by the equality $f(\lambda x^*)= \lambda g(x^*)$ for every $x^* \in K_\mathbb{L}$ and every $\lambda \in [0,1]$.

Notice that $\sup_{x \in \mathbb{L}} (|x_1^*(x)|+|x_2^*(x)|) = 1+\varepsilon$. Thus,
$$\|f\|_{FBL\langle \mathbb{L} \rangle} = \sup \left\{ \sum_{i=1}^n |f(y_i^*)| : n \in \mathbb{N}, \, y_1^*, \ldots, y_n^* \in \mathbb{L}^*, \sup_{y \in \mathbb{L}} \sum_{i=1}^n |y_i^*(y)| \leq 1 \right\}\geq$$
$$\geq \left\vert f\left(\frac{1}{1+\varepsilon}x_1^*\right) \right\vert+ \left\vert f\left(\frac{1}{1+\varepsilon}x_2^*\right) \right\vert=\frac{2}{1+\varepsilon}.$$

\bigskip
We are going to show now that $\|\hat{\imath}(f)\|_{FBL\langle \mathbb{M} \rangle} \leq \frac{1}{1-\varepsilon}$. Suppose that $\|\hat{\imath}(f)\|_{FBL\langle \mathbb{M} \rangle} > \frac{1}{1-\varepsilon}$, and let us obtain a contradiction. 
Notice that each element $x^* \in V_1 \cup V_2$ takes three different values, so it admits a unique extension $\hat x^*$ to $\mathbb{M}$. Furthermore, if $x^* \in V_1 $ then $\hat x^*((1,3))=x^*((1,1))< -1+\frac{\varepsilon}{2}$. Analogously, if $x^* \in V_2 $ then $\hat x^*((1,3))=x^*((3,3))> 1-\frac{\varepsilon}{2}$. In any case, 
\begin{equation}
\label{eq1}
\|\hat x^* \|=\sup_{x \in \mathbb{M}}|\hat x^*(x)|=|\hat x^*((1,3))| > 1-\frac{\varepsilon}{2}>1-\varepsilon.
\end{equation}

Notice that $$\|\hat{\imath}(f)\|_{FBL\langle \mathbb{M} \rangle} = \sup \left\{ \sum_{i=1}^n |f(y_i^*\vert_{\mathbb{L}})| : n \in \mathbb{N}, \, y_1^*, \ldots, y_n^* \in \mathbb{M}^*, \sup_{y \in \mathbb{M}} \sum_{i=1}^n |y_i^*(y)| \leq 1 \right\}.$$

Thus, if $\|\hat{\imath}(f)\|_{FBL\langle \mathbb{M} \rangle} > \frac{1}{1-\varepsilon}$ then there exists $y_1^*, \ldots, y_n^* \in \mathbb{M}^*$ such that $\sum_{i=1}^n |f(y_i^*\vert_{\mathbb{L}})|>\frac{1}{1-\varepsilon}$. Without loss of generality, we can assume that $|f(y_i^*\vert_{\mathbb{L}})|\neq 0$ for every $i\leq n$. By the definition of $f$, this is in turn equivalent to the fact that $\frac{y_i^*\vert_{\mathbb{L}}}{\|y_i^*\vert_{\mathbb{L}}\|} \in V_1 \cup V_2$ for every $i\leq n$. But then, by (\ref{eq1}), 
$$ \frac{|y_i^*((1,3))|}{\|y_i^*\vert_{\mathbb{L}}\|}> 1-\varepsilon \mbox{ for every } i\leq n.$$

Thus,
$$ \sum_{i=1}^n |f(y_i^*\vert_{\mathbb{L}})|=\sum_{i=1}^n \|y_i^*\vert_{\mathbb{L}}\|\left\vert f\left(\frac{|y_i^*|}{\|y_i^*\vert_{\mathbb{L}}\|}\right)\right \vert \leq \sum_{i=1}^n \|y_i^*\vert_{\mathbb{L}}\| <  \sum_{i=1}^n\frac{|y_i^*((1,3))|}{1-\varepsilon} \leq \frac{1}{1-\varepsilon},$$
which yields the desired contradiction.

In conclusion, $\|\hat{\imath}(f)\|_{FBL\langle \mathbb{M} \rangle} \leq \frac{1}{1-\varepsilon}$ and $\|f\|_{FBL\langle \mathbb{L} \rangle} \geq \frac{2}{1+\varepsilon}$. Since $\varepsilon>0$ is arbitrarily small, we conclude that $\hat \imath$ is not an isometry.
\end{proof}

\section*{Acknowledgements}
The authors are deeply grateful to Pedro Tradacete for sharing part of his research in progress and for fruitful conversations on the topic of the paper.

Research partially supported by Fundaci\'{o}n S\'{e}neca [20797/PI/18] and by project MTM2017-86182-P (Government of Spain, AEI/FEDER, EU). The research of G. Mart\'inez-Cervantes
was co-financed by the European Social Fund and the Youth European Initiative under Fundaci\'on S\'eneca
[21319/PDGI/19]. J. D. Rodr\'iguez Abell\'an was also supported by project 20262/FPI/17, Fundaci\'on S\'eneca, Regi\'on de Murcia (Spain). The research of A. Rueda Zoca was also supported by Juan de la Cierva-Formaci\'on fellowship FJC2019-039973, by MICINN (Spain) Grant PGC2018-093794-B-I00 (MCIU, AEI, FEDER, UE), by Junta de Andaluc\'ia Grant A-FQM-484-UGR18 and by Junta de Andaluc\'ia Grant FQM-0185.

\end{document}